\documentclass[reqno, 11pt]{amsart}
\usepackage[letterpaper,hmargin=1in,vmargin=.92in]{geometry}
\usepackage{mathrsfs, amsmath, amssymb, amsthm, verbatim, bbm} 
\usepackage[T1]{fontenc}
\usepackage[hidelinks]{hyperref}

\usepackage{comment}
\usepackage{graphicx}
\usepackage{xcolor,mathabx}
\newtheorem{lem}{Lemma}
\newtheorem{thm}[lem]{Theorem}

\newtheorem{cor}[lem]{Corollary}

\DeclareMathOperator \re {Re}
\DeclareMathOperator \im {Im}

\newcommand{\Real}{\mathbb{R}}
\newcommand{\Integers}{\mathbb{Z}}
\newcommand{\Natural}{\mathbb{N}}
\newcommand{\Complex}{\mathbb{C}}
\newcommand{\loc}{\operatorname{loc}}

\newcommand{\mch}{\mathcal{H}}

\newcommand{\mcd}{\mathcal{D}}

\numberwithin{equation}{section}
\numberwithin{lem}{section}

\newcommand{\Nev}{N_E}
\newcommand{\Ji}{\mathfrak{J}}

\title[From resolvent expansions at zero to long time wave expansions] {From resolvent expansions at zero to long time wave expansions}
\author{T. J. Christiansen, K. Datchev, and M.  Yang}
\address{Department of Mathematics, University of Missouri, Columbia, MO 65211 USA}
\email{christiansent@missouri.edu}
\address{Department of Mathematics, Purdue University, West Lafayette, IN 47907 USA}
\email{kdatchev@purdue.edu}
\address{Department of Mathematics, University of California, Berkeley, CA 97420 USA}
\email{mxyang@math.berkeley.edu}
\begin{document}
\begin{abstract}
We prove a general abstract theorem deducing wave expansions as time goes to infinity from resolvent expansions as energy goes to zero, under an assumption of polynomial boundedness of the resolvent at high energy. We give applications to 
obstacle scattering, to Aharonov--Bohm Hamiltonians, to 
scattering in a sector, and to scattering by a compactly supported potential. 
\end{abstract}
\maketitle

\section{Introduction}
Let $P$ be a self-adjoint operator on a Hilbert space $\mathcal H$ with domain $\mathcal D$. Consider the solution $u=U(t)f:=\frac{\sin{t\sqrt{P}}}{\sqrt{P}}f$ of the wave equation with initial conditions
\begin{equation}
\label{eq:waveIntro}
    \begin{cases}
    (\partial_t^2+P) u=0,\\
    u\rvert_{t=0}=0,\ \partial_tu\rvert_{t=0}=f.
    \end{cases}
\end{equation}

In Theorem \ref{t:main} we will state and prove decay and asymptotics for solutions to \eqref{eq:waveIntro} as $t \to \infty$ under abstract assumptions on the resolvent of $P$, but first we present four applications.  In each case, we give only the leading terms in the expansion, but Theorem \ref{t:main} gives a much more detailed result.

\subsection{Scattering by several convex obstacles}\label{s:convob}

Let $\mathcal{O}=\bigcup_{j=1}^J \mathcal{O}_j$, where the $\mathcal{O}_j$ are open, strictly convex bounded connected subsets of $\Real^2$ with smooth boundary.  Assume in addition that they satisfy the 
Ikawa no-eclipse condition: if $i$, $j$, and $k$ are all different, then $\overline{\mathcal{O}}_i$ does not intersect the convex hull of $\overline{\mathcal{O}}_j\cup \overline{\mathcal{O}}_k.$  Then we consider the wave 
equation in the exterior of $\mathcal{O}$ with Dirichlet boundary conditions, i.e. $P = -\Delta$ with domain $\mathcal D = H^2(\mathbb R^2 \setminus \overline{\mathcal{O}}) \cap H^1_0(\mathbb R^2 \setminus \overline{\mathcal{O}}) $.   

\begin{thm}\label{t:obex}  Let $\chi$, $f \in C_c^\infty(\Real^2\setminus \overline{\mathcal{O}})$, and $u$ be as in \eqref{eq:waveIntro}.  
Let $G\in \mathcal D_{loc}$ satisfy $PG=0$ and  $G(x,y)= \log|(x,y)| +O(1)$ as $(x,y)\rightarrow \infty$.  
Then, for any $s \in \mathbb N$, we have 
$$\left\| \chi \left(u (t) -\frac{1}{2\pi t(\log t)^2}  \left(\int_{\Real^2\setminus{\mathcal{O}} }f G \right) G\right)\right\|_{H^s} = O(t^{-1}(\log t)^{-3})\; \text{as $t\rightarrow \infty$} .$$
\end{thm}

Theorem \ref{t:obex} is proved in Section \ref{s:convob2}, using low energy resolvent results from \cite{cd}, and high energy results from \cite{vac,vac0}.   We remark
that \cite[Theorem 3]{vac} contains a result for local energy decay for solutions of the 
wave equation; our decay rate is slightly better thanks to our more detailed analysis of the resolvent near zero.

\subsection{Aharonov--Bohm Hamiltonians}
\label{s:ab1}
Let
\begin{gather*}
     P = (-i \vec \nabla - \vec A)^2, \qquad \vec A = \sum_{k=1}^n \alpha_k \vec A_0(x-x_k,y-y_k), \quad
     \vec A_0(x,y) = \frac{(-y,x)}{x^2+y^2} = \vec \nabla \arg(x+iy),
\end{gather*}
where $\alpha_k, x_k, y_k\in \Real$. Let $s_k = (x_k,y_k)$ and $S = \{s_1,\dots,s_k\}$ are poles of the vector potential $\Vec{A}$. Assume no three poles are collinear and equip $P$ with its Friedrichs domain $\mathcal D \supset C_c^\infty(\mathbb R^2 \setminus S)$ on $\mathcal H = L^2(\mathbb R^2)$ (See \cite{at,ds,cf,f24} for more  on extensions of the Aharonov--Bohm Hamiltonians).  The form of the wave asymptotics is governed by the value of the \textit{total flux} 
\[
 \beta= \alpha_1+ \cdots + \alpha_n.
\]

\begin{thm} \label{t:wintro}Let $\chi$, $f \in C_c^\infty(\mathbb R^2 \setminus S)$, and $u$ be as in \eqref{eq:waveIntro}.    Then, for any $s \in \mathbb N$, as $t\rightarrow \infty$,
 \begin{enumerate}
\item if $\beta$ is  half an odd integer, then there is a constant $c>0$ so that $\| \chi u(t)\|_{H^s} = O(e^{-c t})$.
\item if $2\beta\not \in \Integers$, let $\mu_m= \min(\beta -\lfloor \beta \rfloor, 1+\lfloor \beta \rfloor -\beta)$ and $\mu_M= \max(\beta -\lfloor \beta \rfloor, 1+\lfloor \beta \rfloor -\beta)$.   Then there is a  function $\tilde{u}\in \mathcal{D}_{\loc}$ such that  
\begin{equation*}
\| \chi (u(t) -\tilde{u}t^{-1-2\mu_m})\|_{H^s} = O(t^{-1-4\mu_m})+ O(t^{-1-2\mu_M}).
\end{equation*}
\item if $\beta \in \Integers$, then there is a  function $\tilde{u}\in \mathcal{D}_{\loc}$ such that 
$$ 
\| \chi (u(t) -\tilde{u}t^{-1}(\log t)^{-2})\|_{H^s} = O(t^{-1}(\log t)^{-3}).$$
\end{enumerate}
\end{thm}
Theorem \ref{t:wintro} is proved in Section \ref{s:ab2}, using low energy resolvent results from \cite{cdy}, and high energy results from \cite{my}.   
 
\subsection{Laplacians on Cones}

Laplacians on cones 
are operators in many respects similar to the Aharonov--Bohm Hamiltonians.  For simplicity, we treat just one family of examples. 

Fix $L>0$ such that $\alpha=\pi/L$ is Diophantine (see Section \ref{s:cones} for the definition); for example we may take $L$ rational, or $\alpha$ irrational algebraic.  Equip $X=(0,\infty)\times(0,L)$ with the metric $g=dr^2+r^2dy^2$. Let $P=-\Delta$ and let  $\mathcal D$  be the Friedrichs domain. This $P$  corresponds to the Dirichlet Laplacian on a sector.


\begin{thm}\label{t:cones}
Let $\chi, \ f \in C_c^\infty(X)$, and $u$ be as in \eqref{eq:waveIntro}. Then, for any $s\in \mathbb N$, as $t \to \infty$,
\[
\| \chi( u(t,r,y) - c_1 t^{-1 - 2\alpha} r^{\alpha} \sin(\alpha y )\|_{H^s} =  O(t^{- 3 - 2 \alpha}) + O(t^{-1 - 4 \alpha}),
\]
where
\[
c_1 =  \frac {2  \Gamma(\frac 1 2 + \alpha) \cos(\pi \alpha)} {\pi^{3/2}\Gamma(\alpha) } \int_0^\infty   \int_0^L f(\tilde r,\tilde y) \sin(\alpha \tilde y) \tilde r ^{1+\alpha}\, d \tilde y d \tilde r, \qquad \alpha = \frac \pi L.
\] 
\end{thm} 

Theorem \ref{t:cones} is proved in Section \ref{s:cones} using Bessel function formulas. The same approach works for more general cones $X=(0,\infty) \times Y$, $g=dr^2 + r^2h$, $(Y,h)$ a compact manifold. In particular we expect Theorem \ref{t:cones} to hold as stated without the convenient simplifying assumption that $\alpha$ is Diophantine. However, if there is an integer $n$ such that $n^2 - (\dim Y -1)^2/4$ is an eigenvalue of $Y$, then the  
resolvent expansion \eqref{e:ajkexp} in the 
hypotheses of Theorem \ref{t:main} contains logarithmic terms. 

More general manifolds with conic ends, or black-box perturbations of these, can also be treated using Theorem \ref{t:main}. In that case low energy resolvent expansions follow as in Section 7C of \cite{MuSt}, and high energy bounds can be obtained under appropriate hypotheses of mild trapping or absence of trapping by \cite{vz,dv,bw13,cdmrl}. For such results, it is the behavior of the operator at infinity that governs the general form of the expansion, not the behavior near any conic singularities; this well-known fact can be seen by applying the arguments of \cite{MuSt} or \cite{cd2} to the general formulas of e.g.~\cite[Theorem 2.1]{by20}. Some such expansions, with corresponding Schr\"odinger asymptotics, can be found in \cite{xpwang}. Foundational results on  Laplacians on cones, including propagation of singulatities of wave equations, were obtained in \cite{ct1,ct2} and \cite{mw1}, 
with further applications in \cite{bm22}.




\subsection{Schr\"{o}dinger operators}
Schr\"{o}dinger operators on $\Real^2$ provide a rich class of examples of operators with interesting resolvent 
behavior near $0$.  Not
only may they have negative eigenvalues, but there are three different possibilities for unbounded behavior of the cut-off resolvent  at $0$, and these are reflected in the behavior of solutions of the wave equation.  

Let $P=-\Delta +V$ on $L^2(\Real^2)$, with $V\in C_c^\infty(\Real^2;\Real)$ and $\mcd=H^2(\Real^2).$ Let 
$E_1<\cdots<E_{\Nev}<0$ denote the negative eigenvalues of $P$, with corresponding projections $\Pi_1,\dots ,\Pi_{\Nev}$.
Set
\begin{equation}
    \label{eq:Gl}
    \mathcal{G}_\ell=\{ g\in H^2_{\loc}|\; Pg=0\; \text{and}\; g(x,y)=O(|(x,y)|^{\ell})\;\text{as $|(x,y)|\rightarrow \infty$}\}.
\end{equation}
A nonzero element of $\mathcal{G}_{-2}$ is an eigenfunction of $P$ with eigenvalue $0$.  An element of 
$\mathcal{G}_{-1}\setminus \mathcal{G}_{-2}$ is sometimes called a $p$-resonant state, while an element of 
$\mathcal{G}_{0}\setminus \mathcal{G}_{-1}$ is called an $s$-resonant state.
\begin{thm}\label{t:schroex}
Let $\chi,\; f\in C_c^\infty(\Real^2),\; V\in C_c^\infty(\Real^2;\Real)$ and let $u$ be as in \eqref{eq:waveIntro}.  Let $s \in \mathbb N$. Then 
$$\left\|\chi u(t) - \chi (u_d(t)+ u_{z,l}(t))\right\|_{H^s} =O(t^{-1 } (\log t)^{-2}).$$
Here
\begin{enumerate}
    \item $u_d(t)=\sum _{\ell=1}^{\Nev} \frac{\sinh(\sqrt{-E_\ell}t)}{\sqrt{-E_\ell}}\Pi_\ell f.$
    \item There is an integer $M$, $0\leq M \leq 2,$  functions $\Ji_m\in C^\infty(\Real_+)$ ($1\leq m\leq M$) and functions $U_{-2},\; U_{\omega_m},$ and $U_0\in H^2_{\loc}(\Real^2)$ such that 
    $$u_{z,l}(t,x,y) = t U_{-2}(x,y)+ \sum_{m=1}^M\Ji _m(t) U_{_{\omega_m}}(x,y)+ t^{-1} U_{0}(x,y).$$
    We have $U_{-2}\in \mathcal{G}_{-2}$, $U_{\omega_m}\in \mathcal{G}_{-1}\setminus \mathcal{G}_{-2}$, and there are $\tilde{U}_0\in (\mathcal{G}_{0}\setminus \mathcal{G}_{-1}) \cup\{0\}$, $\tilde{U}_{-2}\in (\mathcal{G}_{-2}\setminus \mathcal{G}_{-3})\cup\{0\}$ such that $U_0=\tilde{U}_0+\tilde{U}_{-2}.$
    Moreover, $\Ji_m(t)= t((\log t)^{-1}+O((\log t)^{-2}))$ as $t\rightarrow \infty$.
\end{enumerate}
\end{thm}
An explicit integral expression for $\Ji_m$ is given in \eqref{eq:Jm}.

If $\mathcal{G}_0=\{0\}$, then the 
expansion of $u(t)-u_d(t)$ resembles that of Dirichlet 
obstacle case of Theorem \ref{t:obex}.  That is, 
there is a function $G\in H^2_{\loc}(\Real^2)$ with
$(-\Delta +V)G=0$, $G(x,y)=\log|(x,y)|+O(1)$ as $|(x,y)|\rightarrow \infty$, and a constant $c(f)$ such that 
$\|\chi u(t)-\chi (u_d(t) + c(f)G/t(\log t)^2))\|_{\mathcal D^q}  = O(t^{-1}(\log t)^{-3}).$

\subsection{Further background and context}

Wave decay results similar to ours have been much studied for decades. The field is too wide-ranging to survey here. Let us mention the seminal work of Morawetz \cite{m61}, and the surveys and works in \cite[Epilogue]{lp89}, \cite[Chapter X]{va}, \cite{dr}, \cite{tat}, \cite{dz}, \cite{vasy}, \cite{sch21}, \cite{klainerman}, \cite{hin}, \cite{lo24}.

In this paper we revisit and extend some of the abstract results of \cite[Chapter X]{va}, and apply them  to  our examples using the resolvent expansions of \cite{MuSt, cd, cd2, cdy}.

For obstacle scattering, in the case of a non-trapping obstacle, Melrose--Sj\"ostand \cite{mel78} showed that the local energy decay is exponential in odd dimension and polynomial in even dimensions. In the case of trapping, Ikawa (under a no-eclipse condition recalled in Section \ref{s:convob}) showed exponential decay in three dimension, for two obstacles \cite{ika1} and for more obstacles under a dynamical assumption involving the topological pressure of the billiard flow \cite{ika2} (see also \cite{ps}). For general trapping in obstacle scattering, Burq \cite{burq98} showed that the decay rate is logarithmic for compactly supported initial data. In the case of dimension two (and more generally, of even dimensions), one cannot expect exponential decay due to the logarithmic singularity of the resolvent at $0$ and the lack of a strong Huygens principle, even in the absence of trapping.

For the Aharonov-Bohm Hamiltonian \cite{ab}, various wave decay results have only been established for the Hamiltonian with a single pole in \cite{fffp,gk,fzz,gyzz,wang23}, which has scaling and rotational symmetry. Some additional results in settings closer to our Aharonov--Bohm one include \cite{Mur,Kov2,Kov}. To the best of our knowledge, this is the first result on wave decay for the Aharonov--Bohm Hamiltonian with multiple poles.

The paper \cite{vawu} studies questions related to low energy 
behavior of solutions of the wave equation for manifolds which, at
infinity, are like the big end of a cone.
Results for Schr\"{o}dinger operators in various dimensions related to those of Theorem \ref{t:schroex} can be found in \cite{JeKa, je80,Mur, je84,jn, schlag}.

\subsection{Further results} Our methods and results can be readily extended in various directions. 

1. For brevity we compute the leading constants in detail only in Theorems \ref{t:obex} and \ref{t:cones}, but the same approach with more work would give them in Theorems \ref{t:wintro} and \ref{t:schroex} as well.

2. The solution to the more general inhomogeneous Cauchy problem 
\[
    \begin{cases}
    (\partial_t^2+P) u=F,\\
    u\rvert_{t=0}=f_0,\ \partial_tu\rvert_{t=0}=f_1,
    \end{cases}
\]
can be treated in terms of the solution to \eqref{eq:waveIntro} by writing $u(t) = U'(t)f_0 + U(t)f_1 + \int_0^t U(t-s)F(s)ds$.

3. In the theorems above, we consider only initial data in $C_c^\infty$, and estimate remainders in $H^s$, but Theorem \ref{t:main} is more precise and in particular allows us to estimate derivatives of remainders  in terms of derivatives of initial conditions. 

4. The operators $P$ in Theorems \ref{t:obex} and \ref{t:schroex} are special cases of black-box perturbations of the planar Laplacian in the sense of Sj\"ostrand and Zworski \cite{sz} (see also Chapter 4 of \cite{dz}). General low energy resolvent expansions for these are given in \cite{cd2}, and so the results can be extended to settings where there is a high energy resolvent bound analogous to \cite{vac,vac0}. Such bounds have been established in a range of situations of \textit{mild trapping}: see Section 6.6 of \cite{dz} and Section 3.2 of \cite{z17} for surveys of the field. Some more recent progress, as in \cite{vac,vac0}, is based on using a family of monodromy operators \cite{nsz11} to reduce to a one-dimensional fractal uncertainty principle \cite{dz16, bd18, djn21}; see for example a result for semiclassical Schr\"odinger operators in Theorem 3 of \cite{vac0} and one for negatively curved surfaces in \cite{tao24}.

\subsection{Plan of the paper}
In Section \ref{s:main} we state and prove our main abstract theorem, and  in Section~\ref{s:applications} we apply it to prove the theorems above.





\subsection*{Acknowledgements} 
The authors would like to thank Luc Hillairet, Daniel Tataru, and Maciej Zworski for helpful discussions. 
The authors appreciate the helpful comments and corrections of
the referees.
TC and KD are grateful for Simons collaboration grants for mathematicians for travel support. MY is partially supported by the NSF grant DMS-1952939.


\subsection*{List of Notation}
\begin{itemize}
    \item $\Natural_0 =\{0,1,2,\dots\}$ and $\Natural  = \{1,2,3,\dots\}$.
    \item The self-adjoint operator $P$ is defined on a Hilbert space $\mathcal H$ with domain $\mathcal D$,  and for $s  \ge 0 $ we put $\mathcal D^s = (P+i)^{-s} \mathcal H$. The resolvent is given by $R(\lambda)=(P-\lambda^2)^{-1}$, and maps $\mathcal H \to \mathcal H$ when $\im \lambda >0$ and $\lambda^2$ is not in the spectrum of $P$. 
    \item $\sum_{j,k}   = \sum_{j = j_0}^\infty \sum_{k \in \mathbb Z}$, with $j_0$ as in Assumption (4).
    \item The integration contours $\gamma(\delta,c)$, $\Gamma(\delta,c)$ and $\gamma(\delta,\infty)$ are given in Figure \ref{f:contours}.
    \item The spaces $\mathcal{G}_\ell$ of polynomially bounded solutions to $Pu=0$ are given in equation \eqref{eq:Gl}.
    \item The discrete spectral component $u_d(t)$, zero energy component $u_z(t)$ and remainder component $u_r(t)$ in the wave expansion are defined in equation \eqref{e:3us}.
\end{itemize}

\section{The Main Expansion} \label{s:main}

\subsection{The abstract theorem}
 
Let $P$ be a self-adjoint operator on a Hilbert space $\mathcal H$ with domain $\mathcal D$. For $\lambda$ in the upper half plane such that $\lambda^2$ is not in the spectrum of $P$, let $R(\lambda)=(P-\lambda^2)^{-1} \colon \mathcal H \to \mathcal H$. Assume:
\begin{enumerate}
\item The spectrum of $P$ consists of $[0,\infty)$  together with up to finitely many negative eigenvalues $E_1 < \cdots <E_{\Nev}<0$ with corresponding orthogonal projections $\Pi_1$, $\dots$, $\Pi_{\Nev}$. 
\item There exist a bounded $\chi \colon \mathcal H \to \mathcal H$ and a $c_0>0$ such that $\chi R(\lambda) \chi$ 
continues holomorphically from $\{\lambda \colon \im \lambda >0, \re \lambda \ne 0\}$ to $\{\lambda \colon \im \lambda > -c_0, \re \lambda \ne 0\}$. 
\item There exist $p$ and $q \ge 0$ such that $\|\chi R(\lambda) \chi\|_{\mathcal D^p \to \mathcal D^q} \to0 $ uniformly as $|\re \lambda| \to \infty$ with $|\im \lambda | < c_0$. Moreover, if $s > - c_0$ then
\[
\int^\infty_{1}\|\chi R(\lambda+is) \chi \|_{\mathcal D^p \to \mathcal D^q} d \lambda <+\infty.
\]
\item There are operators $\chi A_{j,k} \chi\colon \mathcal D^p \to \mathcal D^q$ (with the indices $p,q$ as in (3)) and numbers $\nu_j \in \mathbb R$, $b_{j,k} \in \mathbb C \setminus i[0,\infty)$, such that
\begin{equation}\label{e:ajkexp}
\chi R(\lambda) \chi = \sum_{j,k} \chi A_{j,k} \chi \lambda^{\nu_j} \log^k(b_{j,k}\lambda), 
\end{equation}
where the  sum $ \sum_{j,k}  = \sum_{j = j_0}^\infty \sum_{k \in \mathbb Z}$
converges absolutely and uniformly
on compact subsets of 
 $$\big\{ \lambda \colon \arg \lambda \in(-\frac \pi 2,\frac{3\pi}2),\;
 |\re \lambda|<c_0,\; |\im \lambda|<c_0\big\}.$$
 
\end{enumerate}

\begin{thm} 
\label{t:main}
With the assumptions and notation above, if $\chi f = f \in \mathcal D^p$, $0<c<c_0$, and $E_\ell<- c^2$ for all $\ell$, then
\begin{equation}\label{e:3us}
 \chi u(t)=  \chi (u_{d}(t) +  u_{z}(t) +  u_{r}(t)),
\end{equation}
where
\begin{enumerate}
\item $u_d$ is the contribution of the discrete part of the spectrum: 
\[u_{d}(t) = \sum_{\ell=1}^{\Nev} \frac{\sinh(\sqrt{-E_\ell}t)}{\sqrt{-E_\ell}} \Pi_\ell f.\]
\item $u_z$ is the zero-energy contribution: 
\begin{equation}\label{e:chiuz}
\chi u_z(t) =  \frac{1}{2\pi}  \sum_{j,k} \chi A_{j,k} \chi f  \lim_{\delta \to 0^+}\int_{\gamma(\delta,c)} e^{-it \lambda}\lambda^{\nu_j} \log^k(b_{j,k} \lambda)\, d\lambda,
\end{equation}
where  $\gamma(\delta,c)$ is as in Figure \ref{f:contours}.
\item $u_r$ is an exponential remainder:
\begin{equation}\label{e:ur}
\|\chi u_{r}(t)\|_{\mathcal D^q} = O(e^{-ct})\|f\|_{\mathcal D^p}.
\end{equation}
\end{enumerate} 
The convergence in \eqref{e:chiuz} and the bound in \eqref{e:ur} are uniform for  $t\in [0,\infty)$.
\end{thm}


\begin{figure}[ht]
\includegraphics[width=5cm]{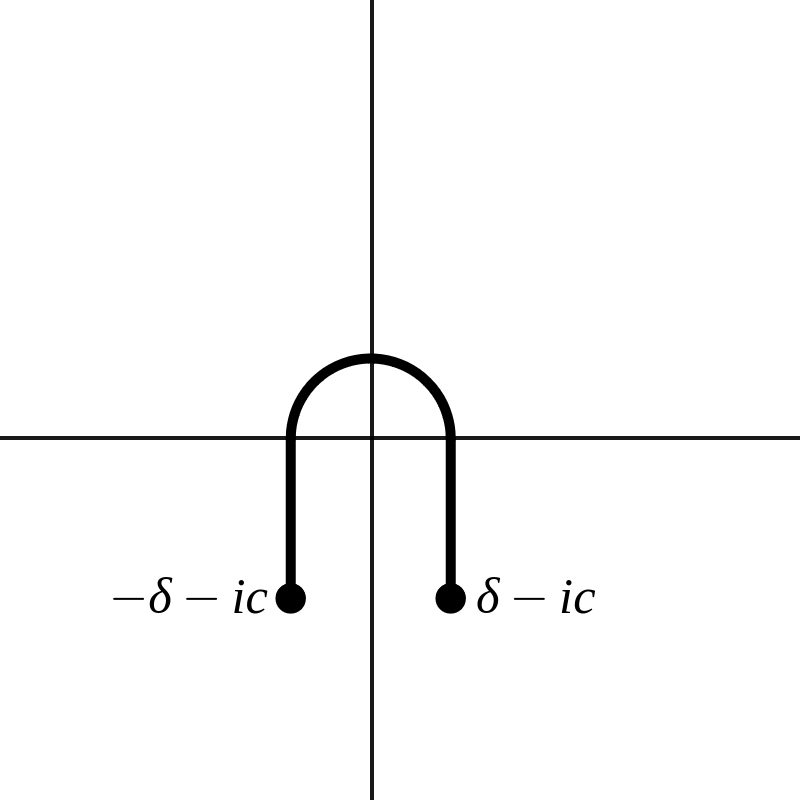}\hspace{5mm}\includegraphics[width=5cm]{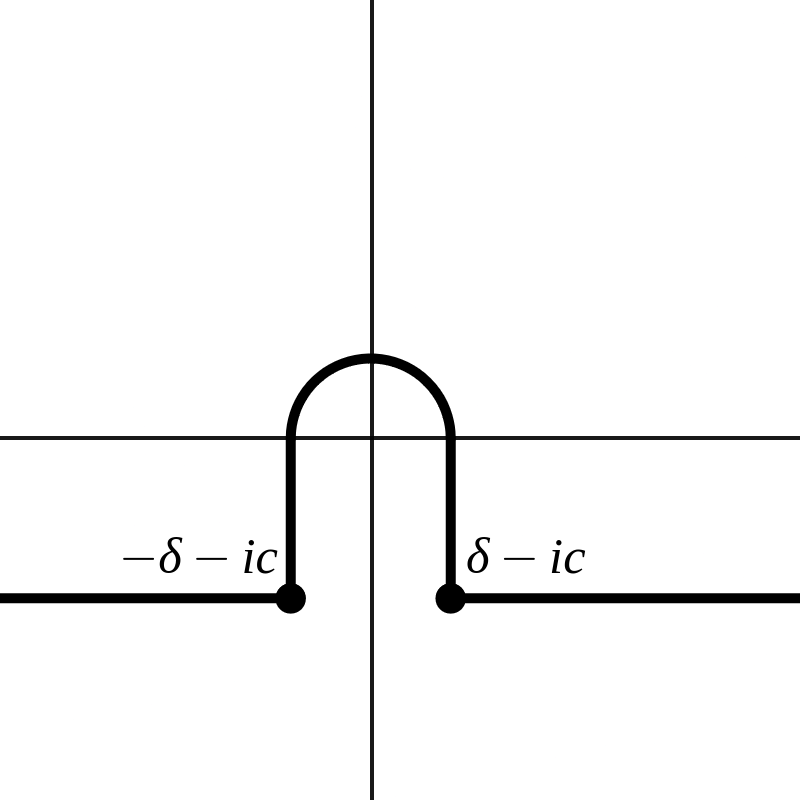}\hspace{5mm}\includegraphics[width=5cm]{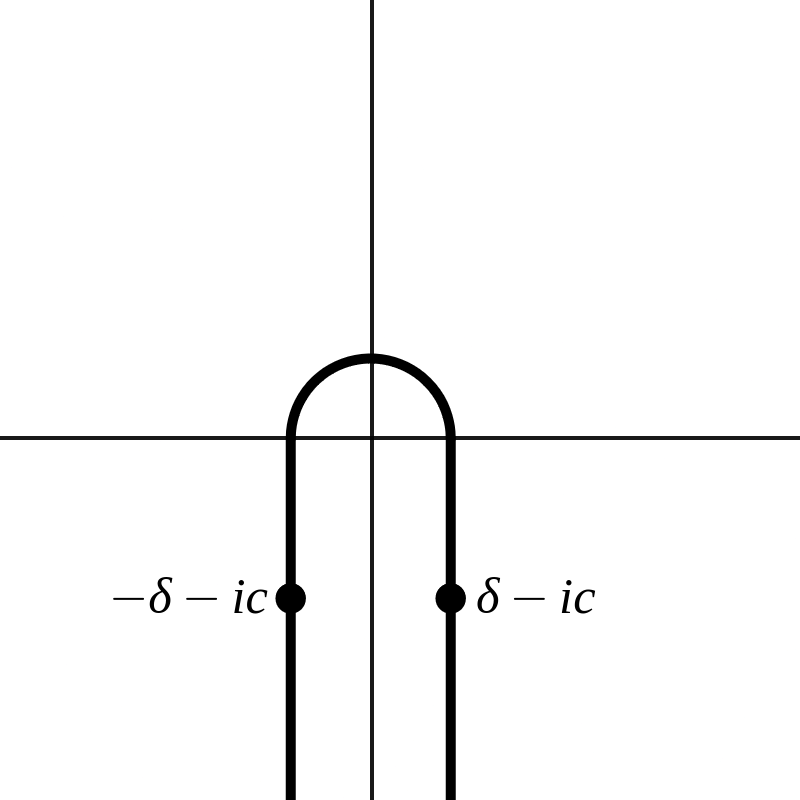}
\caption{The first contour is $\gamma(\delta,c)$, the second is $\Gamma(\delta,c)$, and the third is  $\gamma(\delta,\infty)$. They are all oriented from left to right.}\label{f:contours}
\end{figure}

\noindent{\bf Remark.} The residue theorem allows us to compute the following  $k=0$ terms in   \eqref{e:chiuz}: 
\begin{equation}\label{eq:k=0}
 \frac{1}{2\pi} \lim_{\delta \to 0^+}\int_{\gamma(\delta,c)} e^{-it \lambda}\lambda^{\nu}\, d\lambda =
 \begin{cases}
  - t , \qquad &\nu = -2, \\
  - i, \qquad &\nu = -1, \\
  0, \qquad & \nu = 0, \, 1, \, 2,\dots 
\end{cases}
\end{equation}
That leads to  the following corollary, which is applicable to many cases of odd-dimensional Euclidean scattering and to the half-integer Aharonov--Bohm case.
\begin{cor}
\label{cor:int}
In the expansion \eqref{e:ajkexp}, if $\nu_j =j$ for all $j$, $j_0 = -2$, and $A_{jk} = 0$ whenever $k \ne 0$, then 
\[
 \chi u_z(t) = -t \chi  A_{-2,0}\chi f  -i \chi A _{-1,0} \chi f.
\]
\end{cor}



\begin{proof}[Proof of Theorem \ref{t:main}]

Let $dE_\lambda$ be the spectral measure of $P$ with spectral parameter $\lambda^2$. By 
functional calculus and Stone's formula, using Assumption (1) on the spectrum of $P$, we have
\[\begin{split}
 u(t) - u_{d}(t) &= \int_0^{\infty}\frac{\sin(t\lambda)}{\lambda}fdE_{\lambda} \\
&= \frac{1}{\pi i} \int_0^{\infty}\sin(t\lambda)\left(R(\lambda+ i 0)-R(-\lambda + i 0)\right)f\, d\lambda \\
&= \frac 1 {2\pi}   \int_{-\infty}^\infty e^{-it\lambda }\left(R(\lambda+ i 0)-R(-\lambda+ i 0)\right) f\, d\lambda.     
\end{split}\]
Using $\chi f = f$ and multiplying by $\chi$ on the left, and using Assumptions (2) and (3), we shift the $R(-\lambda + i 0)$ term downward to get
\begin{equation}\label{e:r-l}
\Big\|\int_{-\infty}^{\infty} e^{-it\lambda }\chi  R(-\lambda+i0)\chi f\, d\lambda\Big\|_{\mathcal D^q}  = \Big\|\int_{-\infty}^{\infty} e^{-it(\lambda -ic)}\chi  R(-\lambda+ic)\chi f\, d\lambda\Big\|_{\mathcal D^q} = O(e^{-ct})\|f\|_{\mathcal D^p} .
\end{equation}
That gives
\[
 \chi ( u(t) - u_{d}(t)) =  \frac 1 {2\pi}   \int_{-\infty}^\infty e^{-it\lambda }\chi R(\lambda+i0) \chi f\, d\lambda + O_{\mathcal D^q}(e^{-ct}) .
\]
We shift the remaining integral from $(-\infty, \infty)$ to the contour $\Gamma(\delta,c)$ in Figure \ref{f:contours}, estimate the integrals over $\Gamma(\delta,c) \setminus \gamma(\delta,c)$  as in \eqref{e:r-l}, and take the limit $\delta \to 0^+$ to obtain 
\begin{equation}
\label{e:resint}
    \chi  ( u(t) - u_{d}(t))= \frac 1 {2\pi}  \lim_{\delta \to 0^+}  \int_{\gamma(\delta,c)} e^{-it\lambda }\chi R(\lambda) \chi f\, d\lambda + O_{\mathcal D^q}(e^{-ct}) .
\end{equation}
To conclude we plug in \eqref{e:ajkexp} and use uniform convergence to justify switching the order of the sum and integral and limit.
\end{proof}
 
\subsection{Asymptotics in terms of elementary functions} To describe the behavior of the terms in the main expansion \eqref{e:chiuz} as $t\to \infty$, we follow Lemma 9 of \cite{va2}, which is Lemma 7 of Chapter~X of \cite{va}.  See 
also \cite{erd} and \cite[Lemmas 6.6, 6.7]{Mur} for related
computations.

\begin{lem}\label{l:asylog} Let $\nu \in \mathbb R$, $b \in \mathbb C \setminus i[0,\infty)$, $k \in \mathbb Z$, $0<c'<c$, $t_0>0$. Then 
\begin{equation}\label{e:asy1}
\lim_{\delta \to 0^+}\int_{\gamma(\delta,c)} e^{-it \lambda}\lambda^{\nu} \log^k(b \lambda)\, d\lambda = \lim_{\delta \to 0^+} \int_{\gamma(\delta,\infty)} e^{-it \lambda}\lambda^{\nu} \log^k(b \lambda)\, d\lambda + O(e^{-c't}),
\end{equation}
uniformly for $t \ge t_0$, where $\gamma(\delta,c)$ and $\gamma(\delta,\infty)$ are as in Figure \ref{f:contours}. If $\delta>0$ and $t > 0$, then
\begin{equation}\label{e:asy2}
  \int_{\gamma(\delta,\infty)} e^{-it \lambda}\lambda^{\nu}  \, d\lambda = -2 e^{i \frac \pi 2 \nu} \sin(\pi \nu) \Gamma(\nu+1) t^{-\nu-1} =  e^{i \frac \pi 2 \nu} \frac {2\pi} {\Gamma(-\nu)} t^{-\nu-1}.
\end{equation}
Finally, for any positive integer $M$, we have
\begin{equation}\label{e:asy3}
 \int_{\gamma(\delta,\infty)} e^{-it \lambda}\lambda^{\nu} \log^k(b \lambda)\, d\lambda  = t^{-\nu-1}\sum_{m=0}^M  c_{\nu,k,m}\log^{k-m}t + O(t^{-\nu-1} \log^{k-M-1}t),
\end{equation}
where
\begin{equation}\label{eq:cnukm}
c_{\nu,k,m} =  (-1)^{k-m} \binom k m    \int_{\gamma(\delta,\infty)} e^{-i\mu}\mu^{\nu} \log^m(b \mu)\, d\mu.
\end{equation}
\end{lem}
 
To compute $c_{\nu,k,0}$,  set $t=1$ in \eqref{e:asy2}.    To compute $c_{\nu,k,m}$, do the same after differentiating up to $m$ times with respect to $\nu$. 
In particular, $c_{\nu,k,0} =0$ when $\nu\in\Natural_0$, and $c_{0,-1,1}=2\pi$.

\begin{proof}
Equation \eqref{e:asy1} follows from the exponential decay of the integrand; the integral on the right is independent of $\delta>0$ sufficiently small because the integrand is analytic.

Equation  \eqref{e:asy2} is a version of Hankel's loop integral for the Gamma function. To prove it, first observe that $ \int_{\gamma(\delta,\infty)} e^{-it \lambda}\lambda^{\nu}  \, d\lambda $ is independent of $\delta>0$ by exponential decay and analyticity of the integrand. Next, for $\nu>-1$, taking $\delta \to 0^+$ and parametrizing with $s = it \lambda$ yields
\[
\lim_{\delta \to 0^+}   \int_{\gamma(\delta,\infty)} e^{-it \lambda}\lambda^{\nu}  \, d\lambda = t^{-\nu-1} \Big( e^{3\pi i(\nu+1) /2} \int_{\infty}^0 e^{-s} s^\nu\,ds  + e^{-\pi i(\nu+1) /2} \int_0^\infty e^{-s} s^\nu\,ds  \Big).
\]
Simplifying gives the first equality of \eqref{e:asy2}, and then second comes from the Gamma reflection formula. Then \eqref{e:asy2} follows for other values of $\nu$ by analytic continuation.

To prove \eqref{e:asy3}, we first observe that $ \int_{\gamma(\delta,\infty)} e^{-it \lambda}\lambda^{\nu} \log^k(b \lambda)\, d\lambda $ is independent of $\delta>0$ small enough (we need only stay away from $\lambda = 1/b$ when $k <0$). Then rescale, using $\mu = t \lambda$, to obtain
\[
\lim_{\delta \to 0^+}    \int_{\gamma(\delta,\infty)} e^{-it \lambda}\lambda^{\nu} \log^k(b \lambda)\, d\lambda  =  t^{-\nu-1} \log^kt \lim_{\delta \to 0^+}   \int_{\gamma(\delta,\infty)} e^{-i\mu}\mu^{\nu} (\varepsilon\log b\mu-1)^k\, d\mu,
\]
where $\varepsilon = 1/ \log t$. Next, $g(\varepsilon) = \int_{\gamma(\delta,\infty)} e^{-i\mu}\mu^{\nu} (\varepsilon\log b\mu-1)^k\, d\mu$ is a $C^\infty$ function of $\varepsilon$ for $\varepsilon$ real and small, and is independent of $\delta>0$ small enough. Taylor expanding at $\varepsilon=0$ yields
\[
 \int_{\gamma(\delta,\infty)} e^{-it \lambda}\lambda^{\nu} \log^k(b \lambda)\, d\lambda  =  t^{-\nu-1} \log^kt \Big(\sum_{m=0}^M \frac {g^{(m)}(0)}{m!} \log^{-m}t
 + O(\log^{-M-1}t)\Big),\]
 which implies \eqref{e:asy3} once we compute $g^{(m)}(0)$ by differentiating under the integral sign.
\end{proof}

\subsection{Tail estimate}

We use nonstationary phase to control the tail of a series like the one in \eqref{e:chiuz}:

\begin{lem}\label{l:rem}
Let $c_0>0$, let $U=\{  \lambda \in \Complex:\; i\lambda \not \in[0,\infty) \text{ and }|\lambda|<c_0\}$, and let $(\mathcal X,\|\cdot\|)$ be a Banach space. Let $G \colon U  \to \mathcal X$ be analytic, and suppose $G^{(m)}$ is bounded for $m=0,1,...,N$. Then, for any $c \in (0,c_0)$, 
\[\Big\|  \lim_{\delta \to 0^+}\int_{\gamma(\delta,c)} e^{-it \lambda} G(\lambda)\, d\lambda \Big\| = O(t^{-N}).\]
\end{lem}

\begin{proof} 
Let $C(c)$ be the circle parametrized by $s \mapsto c e^{-is}$, $s \in (-3\pi/2,\pi/2)$. Then
\[
\lim_{\delta \to 0^+}\int_{\gamma(\delta,c)} e^{-it \lambda}   G(\lambda)\, d\lambda = \int_{C(c)} e^{-it \lambda}   G(\lambda)\, d\lambda.
\]
Integrating by parts $N$ times gives
\[ \Big\|\int_{C(c)} e^{-it \lambda}   G(\lambda)\, d\lambda - (it)^{-N} \int_{C(c)}e^{-it\lambda} G^{(N)} (\lambda)\, d\lambda \Big\| =  O(e^{-ct}),\]
where the remainder comes from the boundary terms.
Finally, 
\[
\Big\|\int_{C(c)}e^{-it\lambda} G^{(N)} (\lambda)\, d\lambda \Big\| = \Big\|\lim_{\delta \to 0^+}\int_{\gamma(\delta,c)} e^{-it\lambda}   G^{(N)} (\lambda)\, d\lambda \Big\| = O(1) .
\qedhere\]
\end{proof}

\section{Applications of the main expansion}\label{s:applications}

\subsection{Scattering by convex obstacles}\label{s:convob2}

Let $P$ and $\mathcal D$ be as in Section \ref{s:convob}. Then Assumption (1) holds with $\Nev=0$ because $P \ge 0$, and because the essential spectrum of $P$ equals that of the free Laplacian by \cite[Theorem 9.43]{borthwick}. Assumption (2) holds  by Theorem A of \cite{vac0}.  Here, and in our other examples,
for appropriate smooth, compactly supported $\chi$ we understand
$\chi:\mch \rightarrow \mch$ to be multiplication by $\chi$.

To verify Assumption (3), note that by (3-1) of \cite{vac}, there exist $\beta>0$ and $c_0 >0$ such that
\[
\|\chi R(\lambda) \chi\|_{L^2 \to L^2} = O(\lambda^\beta),
\]
for all $\chi \in C_c^\infty(\Real^2)$ and all $\lambda$ with $|\im \lambda|<c_0$ and $|\re \lambda|$ large enough. Take such a $\chi$ which is $1$ near $\overline{\mathcal O}$, and let $m \in \mathbb N$, $q \in \mathbb N_0$, $\im \lambda>0$. By geometric summation we have the following identity of operators defined on $\mathcal{D}^{q+m}$:
\[
\chi R(\lambda) \chi = -\frac{1}{\lambda^2}\chi\sum_{j=0}^{m-1}\frac{1}{\lambda^{2j}}  P^j \chi + \frac{1}{\lambda^{2m}} \chi (P+i)^{-q} R(\lambda) (P+i)^qP^m \chi .
\]
Now apply $[\chi,(P+i)^{-q}] = (P+i)^{-q} [(P+i)^q,\chi](P+i)^{-q}$ and extend the identity to the lower half plane. Take $\widetilde \chi \in C_c^\infty(\mathbb R^2)$ identically $1$ near the support of $\chi$ and use the boundedness of $[P^q,\chi]$ from $\mathcal{D}^{\alpha}$  to $\mathcal{D}^{\alpha-q+1/2}$  to deduce 
 $$\|\chi R(\lambda)\chi \|_{\mathcal{D}^{m+q}\rightarrow \mathcal{D}^q} \leq C(|\lambda|^{-2}+
|\lambda|^{-2m}\|\tilde{\chi} R(\lambda)\tilde{\chi}\|_{L^2 \rightarrow L^2})\leq C
(|\lambda|^{-2}+|\lambda|^{-2m+\beta}).$$  
By choosing $m\in \Natural$ with
$m>\beta/2+1/2$ and $p=m+q$ we see that assumption (3) is satisfied.

To verify Assumption (4), note that by Theorem 1 of \cite{cd} we have a more precise version of the expansion \eqref{e:ajkexp} with $\mathcal D^p \to \mathcal D^q$ replaced by $\mathcal H \to \mathcal D$. To deduce from this the desired expansion  $\mathcal D^p \to \mathcal D^q$, write
\begin{equation}\label{e:lowid}
R(\lambda) = \sum_{j=0}^{m-1}(I-P+\lambda^2)^j  + R(\lambda)(I-P+\lambda^2)^{m},
\end{equation}
and proceed analogously to the previous paragraph.

Consequently, by Theorem \ref{t:main}, 
\begin{equation}\label{e:chiuconvob}
\Big\| \chi u(t) -  \frac{1}{2\pi}   \sum_{j=0}^\infty   \sum_{k=-j-1}^{j} \chi A_{j,k} \chi f  \lim_{\delta \to 0^+}\int_{\gamma(\delta,c)} e^{-it \lambda}\lambda^{2j} \log^k(b \lambda)\, d\lambda \Big\|_{\mathcal{D}^q} =  O(e^{-c't}) \|f\|_{\mathcal D^p}.
\end{equation} 
Here we have used \cite[Theorem 1]{cd}, which in particular 
bounds the values of $k$ for which $A_{j,k}$ is nonzero and shows that $\arg b =-\pi/2$.  
To simplify the expansion, we bound the tail of the series using the following Lemma.

\begin{lem}\label{l:tail}
Let $J$ be a positive integer and $N=2J-1$. Then
\[\Big\|\sum_{j=J}^\infty   \sum_{k=-j-1}^{j} \chi A_{j,k} \chi f  \lim_{\delta \to 0^+}\int_{\gamma(\delta,c)} e^{-it \lambda}\lambda^{2j} \log^k(b \lambda)\, d\lambda \Big\|_{\mathcal D^q} = O(t^{-N}) \|f\|_{\mathcal D^p}.\]
\end{lem}

\begin{proof}
Use uniform convergence of the sum to switch the order of limit, integral, and sum, to write
\[
\sum_{j=J}^\infty   \sum_{k=-j-1}^{j} \chi A_{j,k} \chi f  \lim_{\delta \to 0^+}\int_{\gamma(\delta,c)} e^{-it \lambda}\lambda^{2j} \log^k(b \lambda)\, d\lambda
=
\lim_{\delta \to 0^+}\int_{\gamma(\delta,c)} e^{-it \lambda}   G(\lambda)\, d\lambda f,
\]
where $ G(\lambda) = \sum_{j=J}^\infty \sum_{k=-j-1}^{j} \chi A_{j,k} \chi  \lambda^{2j} \log^k(b \lambda)$. By termwise differentiation (as in Appendix~B of \cite{cd}), 
\[
\sup\{\|G^{(m)}(\lambda)\|_{\mathcal D^p \to \mathcal D^q} \colon  i\lambda \not \in[0,\infty) \text{ and }|\lambda|<c_0, \ m  = 0, \dots, N\} = O(1).
\]
Thus the hypotheses of Lemma \ref{l:rem} hold, implying the desired conclusion.
\end{proof}

Using  Lemma \ref{l:asylog} and Lemma \ref{l:tail} on \eqref{e:chiuconvob} gives
\[
\Big\|\chi u(t) - \frac 1 {2\pi} \chi A_{0,-1} \chi f \lim_{\delta \to 0^+}\int_{\gamma(\delta,c)} e^{-it \lambda} \log^{-1}(b \lambda)\, d\lambda\Big\|_{\mathcal{D}^q} = O(t^{-3} \log t)\|f\|_{\mathcal D^p},
\] 
which implies, using $c_{0,-1,1} = 2\pi$, that
\[
\Big\|\chi u(t)   - \frac{\chi A_{0,-1} \chi f}{t \log^2 t}   \Big\|_{\mathcal{D}^q} = O(t^{-1} \log^{-3} t)\|f\|_{\mathcal D^p}.
\]
Plugging in the formula for $A_{0,-1}$ from (1-2) of \cite{cd} now gives Theorem \ref{t:obex}.

\subsection{Aharonov--Bohm operators}\label{s:ab2}
We prove Theorem \ref{t:wintro} in this subsection. Let $P$ and $\mathcal D$ be as in Section \ref{s:ab1}. Then Assumption (1) holds with $\Nev=0$ because $P \ge 0$.  Assumption (2) holds by combining \cite[Theorem 1.1]{my} and \cite[Theorems 2, 3]{cdy}. Assumption (3) is satisfied for 
appropriately chosen $c_0$ by \cite[Theorem 1.1]{my} and discussions at the beginning of Section \ref{s:convob2}. Assumption~(4) holds by combining equation \eqref{e:lowid} and \cite[Theorem 2,3]{cdy}, which also gives more precise information about the expansion~\eqref{e:ajkexp}. 

For completeness we briefly recall the form of the low energy expansion in \cite[Theorems 2, 3]{cdy}  of the resolvent of the Aharonov-Bohm Hamiltonian. 
\begin{equation}\label{e:abresexp}
\chi R(\lambda) \chi = 
\begin{cases}
\displaystyle\sum_{j=0}^\infty  \sum_{k=0}^{\infty}  \chi A_{j,k}\chi  \lambda^{2(j+ k \mu_m)}+
\sum_{j=0}^\infty \sum_{k=1}^\infty \chi A_{j,k}' \chi  \lambda^{2(j+ k \mu_M)}, \qquad &\beta \not\in \mathbb Z\\
\displaystyle \sum_{j=0}^\infty   \sum_{k=-j-1}^{j} \chi A_{2j,k} \chi  \lambda^{2j} (\log \lambda-a)^k
, \phantom{\sum^{A^A}}\qquad &\beta \in \mathbb Z
\end{cases}
\end{equation}

With this we turn to proving Theorem \ref{t:wintro}.

\noindent
\emph{Case (1).} If  $\beta$ is half an odd integer, all terms in the expansion \eqref{e:abresexp} are non-negative integer powers of $\lambda$. Corollary \ref{cor:int} thus yields $\chi u_z(t) = 0$ and we obtain (1) of Theorem \ref{t:wintro} by Theorem \ref{t:main}.

\noindent
\emph{Case (2).} 
For $2\beta\not \in \Integers$, by Theorem \ref{t:main}, we have
\begin{equation}
\begin{split}
    \chi u_z(t)  = &  \frac{1}{2\pi}   \sum_{j=0}^\infty   \sum_{k=0}^{\infty} \chi A_{j,k} \chi f  \lim_{\delta \to 0^+}\int_{\gamma(\delta,c)} e^{-it \lambda}\lambda^{2(j+ k \mu_m)} \, d\lambda \\
    & + \frac{1}{2\pi}  \sum_{j=0}^\infty   \sum_{k=1}^{\infty} \chi A'_{j,k} \chi f  \lim_{\delta \to 0^+}\int_{\gamma(\delta,c)} e^{-it \lambda}\lambda^{2(j+ k \mu_M)} \, d\lambda.
\end{split}
\end{equation}
We now bound the tail using the following lemma, which is proved in the same way as Lemma \ref{l:tail}. 
\begin{lem}\label{l:tail2}
Let $N<\min\{2(j_1+k_1\mu_m), 2(j_2+k_2\mu_M)\}$. Then
\begin{equation}
\label{eq:tail2}
    \begin{split}
    \Big\|\sum_{j=j_1}^\infty   \sum_{k=k_1}^{\infty} & \chi A_{j,k} \chi f  \lim_{\delta \to 0^+}\int_{\gamma(\delta,c)} e^{-it \lambda}\lambda^{2(j+ k \mu_m)} \, d\lambda\\
 + & \sum_{j=j_2}^\infty   \sum_{k=k_2}^{\infty} \chi A'_{j,k} \chi f  \lim_{\delta \to 0^+}\int_{\gamma(\delta,c)} e^{-it \lambda}\lambda^{2(j+ k \mu_M)} \, d\lambda \Big\|_{\mathcal D^q}  = O(t^{-N}) \|f\|_{\mathcal D^p} .
    \end{split}
\end{equation}
\end{lem}
With this lemma, we obtain
\[
\Big\|\chi u(t) - \frac 1 {2\pi} \chi A_{0,1} \chi f\lim_{\delta \to 0^+}\int_{\gamma(\delta,c)} e^{-it \lambda} \lambda^{2\mu_m}\, d\lambda\Big\|_{\mathcal D^q} = O(t^{-1-2\mu_{M}})\|f\|_{\mathcal D^p} + O(t^{-1-4\mu_{m}})\|f\|_{\mathcal D^p}.
\] 
When combined with \eqref{e:asy2}, this completes the proof in this case.

\noindent
\emph{Case (3).} 
When $\beta \in \Integers$, this case is  the same as the case of Section \ref{s:convob2}, as the powers in the resolvent expansions are the same.

\subsection{Laplacians on cones} \label{s:cones}

All the assumptions of Theorem \ref{t:main} can be checked using the explicit resolvent formulas in terms of Bessel functions as in \cite[Section 7]{MuSt} or \cite[Theorem 2.1]{by20}. More precisely, if $\varphi(r,y) = \sum_{j=1}^\infty \varphi_j(r)\sin(\alpha jy)$, then
\[
R(\lambda)\varphi(r,y) = \frac {\pi i} {2} \sum_{j=1}^\infty \int_0^\infty J_{\alpha j}(\lambda \min(r,\tilde r)) H^{(1)}_{\alpha j}(\lambda \max(r,\tilde r)) \varphi_j(\tilde r) \tilde r \, d \tilde r\sin(\alpha jy).
\] 

Assumption (1) holds with $N_E=0$.  Assumption 
(2) holds for any 
$\chi \in C_c^\infty(X)$ and $c_0>0$. To check Assumption (3), if $j$ is not large, use the large argument Bessel function expansions of \cite[10.17]{DLMF}. If $j$ is large, then argue as in \cite[Appendix A]{dgs}; the calculations done there for $\lambda\gg 1$ can be directly adapted to $\re \lambda \gg 1$ and $|\im \lambda|$  bounded, because Olver's uniform expansions remain valid \cite[10.20(ii)]{DLMF}. They yield, analogously to \cite[(A.10)]{dgs}, 
\[
|\chi(r) J_{\alpha j}(\lambda \min(r,\tilde r)) H^{(1)}_{\alpha j}(\lambda \max(r,\tilde r)) \chi(\tilde r)| = O (1/\sqrt{j^2 + |\lambda|^2}).
\]
Thus $\|\chi R(\lambda)\chi\|_{L^2 \to L^2}= O(|\lambda|^{-1})$, and Assumption (3) follows as in Section \ref{s:convob2}. 

The expansion in Assumption (4) can be obtained from Bessel function series as in \cite[Section 7]{MuSt}.  The 
condition that $\alpha$ is Diophantine means there is $m$ such that  $|\alpha - \frac n j|^{-1} = O(j^m) $ for any positive integers $n$ and $j$, and hence $|\sin(\pi \alpha j)|^{-1} = O(j^{m-1})$.
 Thus the set
$$\left\{ \left( (2\kappa)^{\alpha j}\sin(\pi \alpha j) \Gamma(1+\alpha j)\right)^{-1}| \; j\in \Natural\right\} $$
is bounded for any fixed $\kappa>0$. (See \cite[Section~6.5]{baker} for Liouville's proof that algebraic numbers are Diophantine, and the rest of that chapter as well as page 8-4 of \cite{milnor} for an introduction to Diophantine approximation. See \cite{sal} for a proof that $\pi$ is Diophantine.) Then it follows
from
 \cite[Theorem 7.6]{MuSt}
that
\[
\chi R(\lambda) \chi = \sum_{j=0}^\infty \sum_{\ell=0}^\infty \chi A_{j,\ell} \chi \lambda^{2 \alpha j + 2 \ell},  \qquad A_{1,0} \varphi(r,y) =  \frac {i \pi(1+ i \cot(\pi \alpha ))r^{\alpha}}{  2^{1+2 \alpha} \Gamma(1+\alpha)^2}  \int_0^\infty \tilde r^{1+\alpha} \varphi_1(\tilde r)\, d \tilde r\sin(\alpha y).
\]  
We note that if instead we had $\alpha\in \mathbb Q$, then there would be logarithmic terms in the expansion of the resolvent as well.


Now we write $f(r,y)=\sum f_j(r)\sin(\alpha j y)$.
Applying Theorem \ref{t:main} and Lemmas \ref{l:asylog} and \ref{l:rem} gives
\[
\Big \| \chi u(t) - \frac 1 {2\pi} \chi A_{1,0} \chi f_1 \lim_{\delta \to 0^+} \int_{\gamma(\delta,c)} e^{-it\lambda} \lambda^{2 \alpha}\, d\lambda \Big\|_{L^2} = O(t^{- 3 - 2 \alpha}) + O(t^{-1 - 4 \alpha}),
\]
which, together with the duplication identity \cite[5.5.5]{DLMF} $ \sqrt \pi \Gamma(1+2\alpha)=2^{2\alpha}\Gamma(\frac 12 + \alpha)\Gamma(1+\alpha)$, implies Theorem \ref{t:cones}.

\subsection{Schr\"{o}dinger operators with compactly supported potentials}\label{s:SO2}

We prove Theorem \ref{t:schroex} in this subsection.

Let $P=-\Delta +V$ on $\mathcal{H}=L^2(\Real^2)$, with $V\in C_c^\infty(\Real^2;\Real)$ and $\mathcal{D}=H^2(\Real^2).$
Considering the hypotheses of Theorem \ref{t:main}, it is well known that (1)-(3) hold, but we outline an argument here.  We write
$R_0(\lambda)$ for the free resolvent, that is, the resolvent of $-\Delta$ on $\Real^2$. Then for $\chi \in C_c^\infty(\Real^2)$ satisfying $\chi V=V$
\begin{equation}\label{eq:BiSc}
    \chi R(\lambda)\chi = \chi R_0(\lambda)\chi (I+VR_0(\lambda)\chi)^{-1}.
\end{equation}
Since $\chi R_0(\lambda)\chi$ has an analytic
continuation to $\{\lambda:\im \lambda>0\}\cup\{ \lambda:\; \re \lambda \not =0\}$, by analytic Fredholm theory the cut-off resolvent
$\chi R(\lambda)\chi$ has a meromorphic continuation to this set. There is a constant $A$ such that in this region 
$\| \chi R_0(\lambda)\chi \|_{L^2 \to L^2} \leq e^{A (\im \lambda)_-}/|\lambda| $, where $(\im \lambda)_-=\max(0, -\im \lambda)$.
Then for $\im \lambda > -C$, $\re \lambda \gg 0$, $\| VR_0(\lambda)\chi\|<1/2$ and 
so \eqref{eq:BiSc} implies $\chi R(\lambda)\chi $ is bounded for such $\lambda$.
 This proves (3).  Combining the fact that $R(\lambda)$
has a meromorphic continuation with \eqref{eq:BiSc} and the expansion of 
the resolvent near $0$ of \cite[Theorem 1]{cd2}  proves Assumptions (1) and (2).
Theorem 1 of \cite{cd2} or the results of 
\cite{jn} show that (4) holds, and give more detailed information about the expansion, which we recall below.

Combining \cite[Theorem 1]{cd2} and \cite[Proposition 4.1]{cd2} yields that, for some $M \in \{0, 1,2\}$, near $\lambda =0$ we have
\begin{equation}\label{eq:soex}
    R(\lambda)=A_{-2,0}\lambda^{-2}+ \left(\sum_{m=1}^M \frac{U_{\omega_m}\otimes U_{\omega_m}}{\log (\lambda b_{-2,-1,m})}\right)\lambda^{-2} + A_{0,1} \log \lambda + \sum_{j=0}^\infty \sum_{k=-\infty}^{k_0(j)} A_{2j,k}\lambda^{2j} (\log \lambda)^k.
\end{equation}
Here $U_{\omega_m}\otimes U_{\omega_m} f= U_{\omega_m}\int f \overline{U}_{\omega_m}dx $, $U_{\omega_m}\in \mathcal{G}_{-1}\setminus \mathcal{G}_{-2}$ and 
$\arg b_{-2,-1,m}=-\pi/2$.
Moreover, $A_{-2,0}$ maps $L^2_c(\Real^2)$ to $\mathcal{G}_{-2}$, $A_{0,1}$ maps $L^2_c(\Real^2)$ to the span of $\mathcal{G}_{-2}\setminus \mathcal{G}_{-3}$ and $\mathcal{G}_{0}\setminus\mathcal{G}_{-1}$, $k_0(j)<\infty$, and $k_0(0)=0.$  

For $0<m\leq M$ set
\begin{equation}\label{eq:Jm}
\Ji_m(t) =\lim_{\delta \rightarrow 0^+}\frac{1}{2\pi}\int_{\gamma(\delta, \infty)}e^{-it\lambda}\lambda^{-2}(\log (\lambda b_{-2,-1,m}))^{-1} d\lambda.
\end{equation}
By \eqref{e:asy3} and \eqref{e:asy2} with $t=1$, $\Ji_m(t)= t( (\log t)^{-1}+O((\log t)^{-2}))$ as $t\rightarrow \infty$. By \cite[ Proposition 5.12]{cd2}, $\Ji_1 = \Ji_2$ whenever $V$ is radial; it is not clear whether or not this is the case more generally.
We will also use the $\nu=-2$ case of \eqref{eq:k=0} 
and 
$$\lim_{\delta \rightarrow 0^+}\int_{\gamma(\delta, c)}e^{-it\lambda} \log \lambda d\lambda = -2\pi t^{-1} +O(e^{-ct}).$$
With these observations and the expansion \eqref{eq:soex}, Theorem \ref{t:schroex} follows from Theorem \ref{t:main}
much as the proof of Theorem \ref{t:obex} in Section \ref{s:convob2}.

\end{document}